\documentclass[12pt]{amsart}
\usepackage{amssymb,amsmath,amsthm,latexsym,mathtools,dsfont}
\usepackage{mathrsfs}
\usepackage[mathscr]{euscript}
\usepackage{mathtools}
\usepackage{mdwlist}
\usepackage{graphicx}
\usepackage{enumerate}
\usepackage[shortlabels]{enumitem}
\usepackage{a4wide}

\usepackage{physics}
\usepackage{color}
\definecolor{amber}{rgb}{1.0, 0.75, 0.0}

\newcommand{\R}{\mathbb{R}}

\newcommand{\N}{\mathbb{N}}

\newcommand{\FF}{\mathscr{F}}

\newcommand{\nn}[1]{{\vert\kern-0.25ex\vert\kern-0.25ex\vert #1 
    \vert\kern-0.25ex\vert\kern-0.25ex\vert}}
\newcommand{\lnn}[1]{{\left\vert\kern-0.25ex\left\vert\kern-0.25ex\left\vert #1 
    \right\vert\kern-0.25ex\right\vert\kern-0.25ex\right\vert}}

\renewcommand{\leq}{\leqslant}

\newcommand{\ou}{%
  \mathrel{%
    \vcenter{\offinterlineskip
      \ialign{##\cr$\forall$\cr\noalign{\kern-1.5pt}$\exists$\cr}%
    }%
  }%
}

\newtheorem{theorem}{Theorem}[section]
\newtheorem*{theoremA}{Theorem A}
\newtheorem{lemma}[theorem]{Lemma}

\newtheorem{corollary}[theorem]{Corollary}

\theoremstyle{definition}

\theoremstyle{remark}
\newtheorem{remark}{Remark}

\title[Large cardinals and filter bases]{Large cardinals and continuity of coordinate functionals of filter bases in Banach spaces}
\author[T.~Kania]{Tomasz Kania}
\address[T.~Kania]{Mathematical Institute\\Czech Academy of Sciences\\\v Zitn\'a 25 \\115 67 Praha 1\\Czech Republic  and  Institute of Mathematics and Computer Science\\ Jagiellonian University\\ {\L}ojasiewicza 6, 30-348 Krak\'{o}w, Poland
}
\email{kania@math.cas.cz, tomasz.marcin.kania@gmail.com}

\author[J.~Swaczyna]{Jaros{\l}aw Swaczyna}
\address[J.~Swaczyna]{Mathematical Institute\\Czech Academy of Sciences\\\v Zitn\'a 25 \\115 67 Praha 1\\Czech Republic  and  Institute of Mathematics, {\L}\'od\'z University of Technology, W\'olcza\'nska 215, 93-005 {\L}\'od\'z, Poland}
\email{swaczyna@math.cas.cz, jswaczyna@wp.pl}

\thanks{The authors acknowledge with thanks support received from GA\v{C}R project 19-07129Y; RVO 67985840.}
\date{\today}

\begin{document}
\maketitle
\begin{abstract}
    Assuming the existence of certain large cardinal numbers, we prove that for every projective filter $\mathscr F$ over the set of natural numbers, $\mathscr{F}$-bases in Banach spaces have continuous coordinate functionals. In particular, this applies to the filter of statistical convergence, thereby we solve a problem by V.~Kadets (at least under the presence of certain large cardinals). In this setting, we recover also a result of Kochanek who proved continuity of coordinate functionals for countably generated filters (\emph{Studia Math.},~2012).
\end{abstract}
\section{Introduction}
Let $\mathscr{F}$ be a filter of subsets of $\mathbb N$, that is, an upwards-closed with respect to the inclusion family of subsets of $\mathbb N$ that is also closed under taking intersections. We will be exclusively interested in filters that contain the Fr\'echet filter that comprises those subsets of the set of natural numbers whose complement is finite. Every such a filter defines naturally a notion of convergence of sequences in a metric space $X$; a~sequence $(x_n)_{n=1}^\infty$ in $X$ converges to $x$ \emph{along the filter} $\mathscr{F}$ (or $\mathscr{F}$-\emph{converges to }$x$, in which case we write $x= \lim_{n, \mathscr{F}} x_n$), whenever for every $\varepsilon > 0$ there is $A\in \mathscr{F}$ such that $d(x, x_n) < \varepsilon$ for every $n\in A$. Clearly, convergence along the Fr\'echet filter rectifies the usual convergence of sequences in a metric space. \smallskip

The notion of a filter is dual to the notion of an ideal of sets---for a given filter, the dual ideal comprises complements of sets from the filter. Filters are often interpreted as families of sets that are considered big, whereas sets from a given ideal may be thought as negligible. Due to the complete interchangeability between filters and ideals, we take the liberty of freely referring to facts about ideals (\emph{e.g.}, \cite{Farah}), when required. \smallskip

For reasons related to approximation in Banach spaces, it is desirable to weaken the notion a Schauder basis to encompass $\mathscr{F}$-convergence with respect to filters $\FF$ over the set of natural numbers (see, \emph{e.g.}, results in \cite{CGK}). A sequence $(x_n)_{n=1}^\infty$ in $X$ is an $\mathscr{F}$\emph{-basis} of $X$, whenever for any $x\in X$ there exists a unique sequence of scalars $(a_n(x))_{n=1}^\infty$ such that 
$$ x = \lim_{n, \mathscr{F}} \sum_{j=1}^n a_j(x) x_j = : \sum_{n, \mathscr{F}} a_n(x) x_n.$$
Usually, for practical reasons, it is additionally assumed in the definition of the $\mathscr{F}$-basis that the linear functionals $x\mapsto a_n(x)$ ($n\in \mathbb N$), that we shall call the \emph{coordinate functionals}, are continuous (see, \emph{e.g.}, \cite{CGK, GanichevKadets}). Conspicuously, every Banach space with an $\mathscr{F}$-basis is separable and spaces with $\FF$-bases whose coordinate functionals are uniformly bounded have the approximation property.\smallskip

During the 4\textsuperscript{th} conference \emph{Integration, Vector Measures, and Related Topics} held in 2011 in Murcia, V.~Kadets recalled the question of whether the hypothesis of continuity of coordinate functionals is redundant in the definition of an $\mathscr{F}$-basis, especially in relation to the filter of statistical convergence:
$$\mathscr{F}_{{\rm st }} = \big\{ A\subseteq \mathbb N\colon \lim_{n\to \infty} \frac{|A\cap \{1, \ldots, n\}|}{n} = 1\big\}, $$
which is of particular interest. Evidently, the problem of automatic continuity of the coordinate functionals is equivalent to continuity of the initial basis projections given by $P_nx = \sum_{k=1}^n a_k(x) x_k$ ($x\in X, n\in \mathbb N$). Kochanek showed in \cite{Kochanek} that for countably-generated filters (more precisely, for filters generated by less than $\mathfrak p$ sets, where $\mathfrak p$ is the so-called \emph{pseudo-intersection number}), the coordinate functionals are automatically continuous by establishing uniform boundedness of the initial projections associated to an $\FF$-basis (under the Continuum hypothesis, $\mathfrak{p}=\mathfrak{c}$, the continuum).\smallskip

In the case of usual Schauder bases (that is, bases with respect to the Fr\'echet filter) the standard proofs invoke the uniform boundedness principle/open mapping theorem in order to conclude that the initial basis projections are uniformly bounded. However, already for the filter of statistical convergence, there exist examples of $\mathscr{F}_{{\rm st}}$-bases in a Hilbert space whose initial projections are bounded, but not uniformly (\cite[Example 1]{Kochanek}). Consequently, any proof of automatic continuity of coordinate functionals associated to \emph{every} $\mathscr{F}$-basis would at the same time constitute a~genuinely new proof of continuity of evaluation functionals associated to Schauder bases. \smallskip 

In the present paper, we show that under the assumption of the existence of sufficiently large cardinal numbers that guarantee certain regularity properties of projective subsets of Polish spaces, the coordinate functionals of $\mathscr{F}$-bases with respect to projective filters (so, in particular, the filter of statistical convergence) are continuous. (Any subset of $\mathbb N$ may be naturally identified with its indicator function, which is an element of the Cantor set $\Delta = \{0,1\}^{\mathbb N}$. Thus, a filter on $\mathbb N$ may be viewed as a subset of $\Delta$--in particular, it may or may not be a projective subset thereof.)\smallskip

Our main result reads as follows:

\begin{theoremA}Assume \emph{\textsf{ZFC + }}there is a super-compact cardinal\footnote{Alternatively, by \cite[Conclusion 4.3]{ShelahWoodin}, we may assume \textsf{ZFC +} there exist infinitely many Woodin cardinals followed by a measurable cardinal.}. Let $\mathscr{F}$ be a filter that is a projective subset of the Cantor set. Then, the coordinate functionals associated to every $\mathscr{F}$-basis in a Banach space are continuous.
\end{theoremA}

It is known that the filter of statistical convergence is $F_{\sigma \delta}$ (see, \emph{e.g.}, \cite[Lemma 1.2.2 and p.~9]{Farah}), hence Borel, so in particular projective. We may thus record separately the following corollary.

\begin{corollary}In the theory assumed in Theorem A, the coordinate functionals associated to every $\mathscr{F}_{{\rm st}}$-basis in a Banach space, that is a filter basis with respect to the filter of statistical convergence, are continuous.
\end{corollary}

There are only three isomorphic types of countably generated (proper) filters over $\mathbb N$ that contain the Fr\'echet filter $\text{Fr}$: $\text{Fr}$ itself, $\text{Fr}\oplus \Delta$, and $\text{Fr}\otimes \Delta$, where the latter two are, respectively the Fubini sum and the Fubini product of $\text{Fr}$ and the power-set of the natural numbers, which one may verify directly. (Here, two filters $\FF_1$ and $\FF_2$ over sets $X$ and $Y$ are isomorphic when there exists a bijection $f\colon X \to Y$ such that $A \in \FF_1 \Leftrightarrow f[A] \in \FF_2$; \emph{cf.}~\cite[Lemma 1.2.8]{Farah}, where the notion of a filter isomorphism is slightly different (\cite[Definition 1.2.7]{Farah}).
All the above-mentioned filters are Borel---every countably generated filter is thus Borel as filter isomorphisms (implemented by autobijections of $\mathbb N$) induce self-homeomorphisms of the Cantor set. Consequently, having noticed that countably generated filters are Borel, we recover the main result of Kochanek 
\cite{Kochanek} (at least when $\mathfrak{p}=\aleph_1$), although in a~theory stronger than \textsf{ZFC}.

\begin{corollary}Let $\FF$ be a countably generated filter. In the theory assumed in Theorem A, the coordinate functionals associated to every $\mathscr{F}$-basis in a Banach space are continuous.
\end{corollary}


We shall refrain ourselves from defining  super-compact or Woodin cardinals precisely, instead referring the reader to \cite[Chapter 34]{Jech}. Readers unfamiliar with large-cardinal axioms may think of these as a way of expanding the familiar theory \textsf{ZF} in a way guaranteeing the absoluteness of all projective sentences. \smallskip 

It is to be noted that the assumption on the existence of large cardinals considered in the present paper is stronger than classical axioms that are usually added to \textsf{ZFC} such as \textsf{CH}, \textsf{MA}, or assumptions concerning certain cardinal invariants. Namely, even though, by G\"odel's theorem, we can not prove that \textsf{ZF} is consistent, it is provable that, \emph{e.g.}, `\textsf{ZF} is consistent if and only if so is \textsf{ZFC} + \textsf{CH}'. Analogous claims remain true after replacing \textsf{CH} with other standard independent assertions (\emph{e.g.}, statements obtainable from any \textsf{ZFC} model via forcing). For large cardinal axioms the situation is a bit different, namely if `\textsf{ZF} + existence of certain large cardinal' is consistent, then clearly so is \textsf{ZF}, however the converse implication can not be proved (unless \textsf{ZF} is inconsistent). Nonetheless, large cardinal axioms are widely used especially in relation to Descriptive set theory.

\section{Preliminaries}

\subsection{Inner models of \textsf{ZF}(\textsf{C}) and projective sets in Polish spaces} By an inner model of $\textsf{ZF}$, we understand a subclass of von Neumann's universe of all sets $V$ (the class of hereditarily well-founded sets) that contains the class of all ordinal numbers and which satisfies the axioms of the Zermelo--Fraenkel set theory $\textsf{ZF}$.\smallskip

We define recursively the following sets. We set $L_0(\mathbb R)=\mathbb R$. When $\alpha = \beta +1$ is a~successor ordinal, $L_\alpha(\mathbb R)$ comprises all sets in $L_\beta(\mathbb R)$ together with all subsets of $L_\beta(\mathbb R)$ that are definable from first-order formulae with parameters. For $\alpha$ being a limit ordinal, we set $L_\alpha(\mathbb R)=\bigcup_{\beta < \alpha} L_\beta(\mathbb R)$. Finally, we define the class $L(\mathbb R)$ to be the union of $L_\alpha(\mathbb R)$, where $\alpha$ ranges the class of all ordinal numbers. Roughly speaking, $L(\mathbb R)$ is the smallest (inner) model of set theory that contains the real numbers and all ordinal numbers. It is well known that $L(\mathbb R)$ is an inner model of \textsf{ZF} (however \textsf{AC}, the Axiom of Choice, fails in $L(\mathbb R)$).

\subsubsection{Projective sets}
A subset $A$ of a Polish space $X$ (completely metrisable separable topological space) is \emph{projective}, when it belongs to one of the (boldface) families $\mathbf \Sigma^1_n$ for some $n$, where $\mathbf\Sigma^1_1$ is the family of analytic sets (continuous images of Borel sets in Polish spaces), $\mathbf\Pi^1_1$ is the family of sets whose complements are analytic, and $A$ is in $\mathbf\Sigma^1_{n+1}$, whenever there exists a Polish space $Y$ and a $\mathbf\Pi^1_n$ set $B$ in $X\times Y$ such that $A = \pi_X(B)$, where $\pi_X$ is the projection onto the first coordinate. Projective subsets of a Polish space $X$ are in correspondence with certain formulae allowing for quantification over the space. For $n\in \mathbb N$, we say that a formula of the form
$$\psi(x) \equiv \exists_{f_1 \in X}\; \forall_{f_2\in X}\; \ldots \; \ou{}_{f_n\in X} \Phi(f_1, f_2, \ldots, f_n, x, g)$$
is \emph{projective} (and $\mathbf \Sigma^1_n$), where $\Phi$ is a formula in the language of arithmetic,  possibly contains quantifiers over a~countable set, and $g\in X$ is an arbitrary parameter. More precisely, a~set $A\subset X$ is $\mathbf{\Sigma}^1_n$ if and only if there exists a $\mathbf \Sigma^1_n$ formula $\psi$ such that $A = \{x\in X\colon \psi(x)\}$. Analogously, a similar correspondence exists between $\mathbf \Pi^1_n$-sets and $\mathbf \Pi^1_n$-formulae.\smallskip

For more details concerning projective sets see, \emph{e.g.}, \cite[Chapter V]{Kechris}. For interrelation between projective sets and large cardinals, we refer the reader to \cite[Chapter 32]{Jech}.\smallskip

Projective sets (in Euclidean spaces) were first introduced by Lusin, who investigated whether they are Lebesgue-measurable and/or have the Baire property (that is, whether they differ from an open set by a meagre set). This is indeed so in the so-called Solovay model. We shall require a similar result by Shelah and Woodin (\cite[Theorem 1.7]{ShelahWoodin}, see also \cite[Corollary 34.7]{Jech}), which asserts that under a suitable large cardinal hypothesis, every projective set in a Polish space in $L(\mathbb R)$ has the Baire property and is Lebesgue-measurable.

\begin{theorem}[Shelah--Woodin]\label{swtheorem} Assume that there exists a super-compact cardinal number (or, there exist infinitely many Woodin cardinals followed by a measurable cardinal). Then,\begin{itemize}
    \item every projective subset of $\mathbb R$ is in $L(\mathbb R)$,
    \item every subset of $\mathbb R$ that is in $L(\mathbb R)$ has the Baire property,
    \item every subset of $\mathbb R$ that is in $L(\mathbb R)$ is Lebesgue-measurable.\end{itemize} 
\end{theorem}
Garnir noticed that if in a model of \textsf{ZF}+\textsf{DC} every subset of the real numbers is Lebesgue-measurable, then every linear map between Fr\'echet spaces (in particular, Banach spaces) is continuous (\cite[pp.~195--198]{garnir}). A similar conclusion was independently made by Wright  \cite{Wright}, who used the Baire property of all subsets of Polish spaces in a given model. \smallskip
Let us then record the above remark formally.
\begin{corollary}\label{garnirwright}Assume that there exists a super-compact cardinal number (or, there exist infinitely many Woodin cardinals followed by a measurable cardinal). Then, every linear map between Fr\'echet spaces in $L(\mathbb R)$ is continuous.\end{corollary}

A formula $\psi$ in the language of \textsf{ZF} is $\Delta_0$, whenever it does not involve unbounded quantifiers (see \cite[Definition 12.8]{Jech} for a more formal framework). Such formulae are absolute across all inner models of \textsf{ZF} (\cite[Lemma 12.8]{Jech}). Consequently, $V \models \psi$ if and only if $L(\mathbb R) \models \psi$. Consequently, if $V$ proves \textsf{DC}, the Axiom of dependent choices (a statement asserting that all binary relations with no finite descending chains are well-founded), then so does $L(\mathbb R)$.\smallskip

Let us then record the following simple remark, which will be of particular importance in the proof of the main result.
\begin{remark}\label{mainremark}Every $\mathbf \Pi^1_n$- or $\mathbf \Sigma^1_n$-formula concerning Polish spaces is $\Delta_0$, because it involves only bounded quantifiers. Consequently, $V$ and $L(\mathbb R)$ agree on every $\mathbf \Pi^1_n$ (and $\mathbf \Sigma^1_n$) formula. Thus, under the assumptions of Theorem~\ref{swtheorem}, all projective sets in Polish spaces in $L(\mathbb R)$ have the Baire property, which means that this is also the case in $V$, regardless of the Axiom of choice. \end{remark}



\subsection{Spaces of separable Banach spaces}
Our notation concerning Banach spaces is standard. We consider Banach spaces over the scalar field $\mathbb K$ being either the field of real or complex numbers.\smallskip

It is a widely known fact, sometimes referred as the Banach--Mazur theorem, that the space of continuous functions on the Cantor set endowed with the supremum norm, $C(\Delta)$, is linearly isometrically universal for the class of all separable Banach spaces. Perhaps it is worthwhile mentioning that the Banach--Mazur theorem is true in $\mathsf{ZF} + \mathsf{AC}_\omega$ (that is, $\mathsf{ZF}$ and the Axiom of countable choice; it is immediate that $\mathsf{AC}_\omega$ is a~consequence of $\mathsf{ZF}+\mathsf{DC}$). For the sake of completeness, let us the sketch the proof here.\smallskip

\begin{proof}[Sketch of the proof of the Banach--Mazur theorem in $\mathsf{ZF} + \mathsf{AC}_\omega$] Dodu and Morillon made a remark (see \cite[p.~312]{DoduMorillon}) that on the ground of \textsf{ZF}, if a~Banach space contains a dense well-orderable subset, then the Hahn--Banach theorem is valid for $X$ (that is, every continuous linear functional defined on a subspace of $X$ extends to a continuous linear functional on $X$ with the same norm). (In particular, norm-one continuous linear functionals on $X$ separate points in $X$.)  It follows that $\mathsf{AC}_\omega$ is sufficient for proving the Hahn--Banach theorem for separable Banach spaces. We observe that the Banach--Alaoglu theorem for separable Banach space is constructive (\cite[Theorem 3.2.1]{BS}), so we may conclude that the closed unit ball $B_{X^*}$ of the dual space $X^*$ is weak* (sequentially) compact. Consequently, the mapping $T\colon X\to C(B_{X^*})$ given by $(Tx)(f) = \langle f,x\rangle$ ($x\in X, f\in B_{X^*}$) is isometric. Finally, we observe that the fact asserting that every compact metric space is a continuous image of $\Delta$ is provable in $\mathsf{ZF}+\mathsf{AC}_\omega$, so there exists a continuous surjection $h\colon \Delta\to B_{X^*}$, which induces a linear isometric embedding $J_h\colon C(B_{X^*})\to C(\Delta)$ given by $J_hf = f \circ h$ ($f\in C(B_{X^*})$). This shows that $J_hT$ is an isometric embedding of $X$ into $C(\Delta)$.\end{proof}

Let $\mathcal{F}(C(\Delta))$ denote the hyperspace comprising all non-empty closed subsets of $C(\Delta)$. Then $\mathcal{F}(C(\Delta))$ may be endowed with the Effros--Borel structure making it a standard Borel space with the (Effros--Borel) $\sigma$-algebra generated by the sets
$$E^+(U) :=\{ F\in \mathcal{F}(C(\Delta))\colon F\cap U \neq \varnothing\} \qquad (U\subseteq C(\Delta)\text{ non-empty open set}).$$
 Apparently, there is no canonical Polish space topology on $\mathcal{F}(C(\Delta))$ making the Effros--Borel $\sigma$-algebra its Borel $\sigma$-algebra. Following Godefroy and Saint-Raymond \cite{GoSR}, we shall call every Polish topology on $\mathcal{F}(C(\Delta))$ \emph{admissible} as long as the sets of the form $E^+(U)$ are open with respect to that topology and whose every member may be written as a union of countably many sets of the form $E^+(U)\setminus E^+(V),$ where $U$ and $V$ are open in $C(\Delta)$. Clearly, the Borel $\sigma$-algebra of any admissible topology coincides with the Effros--Borel $\sigma$-algebra. As proved in \cite{GoSR}, admissible topologies in $\mathcal{F}(C(\Delta))$ are plentiful. Moreover, for any countable ordinal number $\gamma$, the classes of $\Sigma^0_\gamma$- and $\Pi^0_\gamma$-sets do not depend on the particular choice of an admissible topology. As the relation `$x\in F$' is $\Pi^0_2$ (\cite[Section 2]{GoSR}), it may be additionally requested that the set
 \begin{equation}\label{closed}
     \{(F, x)\colon F\in \mathcal{F}(C(\Delta)), x\in F\}
 \end{equation}
 is closed in $\mathcal{F}(C(\Delta))\times C(\Delta)$. Henceforth, we shall be using admissible topologies on $\mathcal{F}(C(\Delta))$ that meet this requirement.\medskip
 
 It turns out that the set $\text{SB}$ comprising all closed linear subspaces of $C(\Delta)$ is $\Pi^0_2$ in $\mathcal{F}(C(\Delta))$ and, as such, the relative topology on $\text{SB}$ is Polish (\cite[Section 3]{GoSR}). We refer the reader to \cite{Cuth} for a thoughtful study of admissible topologies on ${\rm SB}$.

\section{Proof of Theorem A}
We shall require the following simple estimation concerning complexity of convergence of $\FF$-series in  (separable) Banach spaces with respect to projective filters.

\begin{lemma}\label{fconvergence}
Let $X$ be a separable Banach space and let $\FF$ be a projective filter on $\mathbb N$ of class $\mathbf \Pi^1_n$. Suppose that $(z_k)_{k=1}^\infty$ is a sequence in $X$. Then, the following formula is $\mathbf \Pi^1_{n}$:
$$\varphi( (a_k)_{k=1}^\infty, z) \equiv \sum\limits_{j,\FF} a_k z_k = z.$$

\end{lemma}
\begin{proof}
Let us consider the set 
$$
Y:=\{ \big((a_k)_{k=1}^\infty,z \big) \in \mathbb K^\N \times X : \sum_{j,\FF}a_j z_j =z \}.
$$
We observe that $$\big((a_k)_{k=1}^\infty,z \big) \in Y \Leftrightarrow \forall_{m \in \N}  \{ j \in \mathbb N\colon \Big\| \sum_{i=1}^j a_i z_i - z \Big\| < \tfrac{1}{m}  \} \in \FF.$$
Let us consider the (continuous) function $f\colon X^{\mathbb 
N}\times X \times \N \to \R$ given by $$f((a_k)_{k=1}^\infty,z,j)=\Big\| \sum_{i=1}^j a_i z_i - z \Big\|.$$ We see that
$$\big((a_k)_{k=1}^\infty,z \big) \in Y \Leftrightarrow \forall_{m \in \N}\; f^{-1} \big[(-\infty, \tfrac{1}{m})\big]_{((a_k)_{k=1}^\infty, z)} \in \FF,$$
so it is enough to show that for any fixed $m\in \N$ set $$Y_m := \{ ((a_k)_{k=1}^\infty,z): f^{-1} \big[(-\infty, \tfrac{1}{m})\big]_{((a_k)_{k=1}^\infty, z)} \in \FF \}$$ is $\mathbf \Pi^1_n$ in $X \times X^\N$. In order to do so, let us consider the function $g: X \times X^\N \to \Delta$ given by $$g(((a_k)_{k=1}^\infty,z)=f^{-1} \big[(-\infty, \tfrac{1}{m})\big]_{((a_k)_{k=1}^\infty, z)}.$$
It is straightforward to check that by openness of set $f^{-1} \big[(-\infty, \tfrac{1}{m})\big]$, the function $g$ is continuous. Since $Y_m = g^{-1}(\FF)$ and projective classes are closed under continuous preimages, we get that $Y$ is of class $\mathbf \Pi^1_{n}$.\end{proof}

\subsection{Proof strategy} 
The proof of Theorem A is divided into two parts. Firstly, we note that, by  Corollary~\ref{garnirwright}, under the existence of a super-compact cardinal (or, infinitely many Woodin cardinals followed by a measurable cardinal), all linear maps between (separable) Banach spaces in $L(\mathbb R)$ are continuous. Thus, if a Banach space in $L(\R)$ has an $\FF$-basis, then then the coordinate functionals associated to that basis are continuous. \medskip

Secondly, we notice that for a projective filter $\mathscr F\subset \Delta$ the formula 
\begin{quote}\emph{If a Banach space $X$ admits an $\FF$-basis, then the coordinate functionals associated to that basis are continuous.}\end{quote}
is actually projective, and as such, absolute between $L(\mathbb R)$ and $V$. Being true in $L(\mathbb R)$, it must be true in $V$ too, where the Axiom of choice holds, if we wish to assume it (\emph{cf}. Remark~\ref{mainremark}).
We are now ready to prove the main theorem in full detail.


\begin{proof}[Proof of Theorem A]We fix an admissible Polish topology on $\text{SB}$, the space of closed linear subspaces of $C(\Delta)$ so that its Borel $\sigma$-algebra agrees with the Effros--Borel $\sigma$-algebra. The ambient Polish space we will be working with is
$$Z= \text{SB} \times C(\Delta) \times C(\Delta)^{\mathbb N} \times \mathbb K^{\mathbb N} \times \mathbb N^{\mathbb N}.$$

For brevity of notation, we require to introduce some conventions concerning quantifiers. Every factor subspace of $Z$ comes with a naturally distinguished point; these are respectively:
$$\{0\}, 0, (0,0,\ldots), (0,0,\ldots), (1,1,\ldots).$$ 
This allows us to identify each factor with the closed subspace of $Z$ comprising the product of singletons being the above-listed distinguished points that are not in the given factor and the factor itself. For example, we identify $C(\Delta)$ with
$$\{\{0\}\}\times C(\Delta) \times \{(0,0,\ldots)\}\times \{ (0,0,\ldots)\} \times \{(1,1,\ldots)\}$$
and so on. \smallskip

For a Borel subset $F$ of $Z$ and a formula $\psi$, the expressions $\forall_{z\in F}\psi(z)$ (respectively, $\exists_{z\in F} \psi(z)$) mean $\forall_{z\in Z} (z\in Z \setminus F \vee \psi(z))$ ($\exists_{z\in Z} (z\in F \wedge \psi(z))$). Moreover, we will require making statements about separable Banach spaces themselves (that is, points in $\text{SB}$) and this needs further explanation indeed.\smallskip

For example, for statements of the form
$$\ou_{X\in {\rm SB}} \ou_{(x_n)_{n=1}^\infty\in X} \Phi(X,(x_n)_{n=1}^\infty),$$
where $\Phi$ is some formula, have to be expressed as statements whose quantifiers are bounded by $Z= \text{SB} \times C(\Delta)^{\mathbb N} \times C(\Delta) \times \mathbb K^{\mathbb N} \times \mathbb N^{\mathbb N}$ only. As originally we quantify over
$$
\text{SB} \times X^{\mathbb N} \times X \times \mathbb K^{\mathbb N} \times \mathbb N^{\mathbb N},
$$
where $\forall_{(x_n)_{n=1}^\infty\in X}$ formally depends on $X\in {\rm SB}$, we may express the latter set as 
$$\Big(\bigcap_{n \in \N } \text{Ev} \times C(\Delta)^{\N \setminus \{n\}} \times C(\Delta)  \times \mathbb K^{\mathbb N} \times \mathbb N^{\mathbb N}\Big) \cap \big(\text{Ev} \times C(\Delta)^{\N } \times C(\Delta)\times \mathbb K^{\mathbb N} \times \mathbb N^{\mathbb N}\big),$$
where $\text{Ev}:=\{(X,x)\colon X\in \text{SB},\, x\in X\}$
is, by \eqref{closed}, a relatively closed subset of the product space $\text{SB}\times C(\Delta)$. Finally, we may apply the above-described conventions concerning quantification over subsets to ensure that all quantifiers are over $Z$.
\smallskip


%



We are now ready to formalise our formula asserting that if a separable Banach space has an $\FF$-basis, then the coordinate functionals are continuous:

$$\begin{array}{l} \forall_{ X \in {\rm SB}}\forall_{ (x_k)_{k=1}^\infty \in X^\N} \Big[ \big(\forall_{y \in X}\; \exists!_{(a_k)_{k=1}^\infty \in \mathbb K^\N}\;\sum\limits_{k,\FF} a_k x_k = y \big)  \Rightarrow \\

\Rightarrow \big(\exists_{(M_k)_{k=1}^\infty \in \N^\N}\;\forall_{y \in X}\; \exists_{(a_k)_{k=1}^\infty \in \mathbb K^\N}\;\sum\limits_{k,\FF} a_k x_k = y \wedge |a_k|\leq \|y\|\cdot M_k\big)\Big]
\end{array}$$
Equivalently,
$$\begin{array}{l} \forall_{ X \in {\rm SB}}\forall_{ (x_k)_{k=1}^\infty \in X^\N} \Big[ \neg\big(\forall_{y \in X}\; \exists!_{(a_k)_{k=1}^\infty \in \mathbb K^\N}\;\sum\limits_{k,\FF} a_k x_k = y \big)  \vee\\

\vee \big(\exists_{(M_k)_{k=1}^\infty \in \N^\N}\;\forall_{y \in X}\; \exists_{(a_k)_{k=1}^\infty \in \mathbb K^\N}\;\sum\limits_{k,\FF} a_k x_k = y \wedge |a_k|\leq \|y\|\cdot M_k\big)\Big].
\end{array}$$

By Lemma~\ref{fconvergence}, the expression $\sum_{k,\FF} a_k x_k = y$ does not contribute to the complexity of the above formula, so by a direct count of the quantifiers, we infer that the formula is projective, being  $\mathbf \Pi^1_{n+4}$, when $\FF$ is $\mathbf \Pi^1_{n}$.
\end{proof}

\begin{remark}
One may be rightfully apprehensive that the tools used along the way yield a~result that is \emph{too good to be true} and try to achieve by similar methods conclusions that are plainly false in \textsf{ZFC} in order to refute it. For example, it is desirable to attempt obtaining continuity of \emph{all} linear functionals on separable Banach spaces in this way. Such a reasoning would likely go along the following lines: in \textsf{ZF}, assuming that all projective sets in separable Banach space have the Baire property, all linear functionals are automatically continuous. The formula defining continuity is projective, thus by the absoluteness argument, it must be true in $V$. This however is incorrect.\smallskip

We want to point out that one must make sure that all spaces used to quantify the formula are Polish, however (in \textsf{ZFC}) the algebraic dual of an infinite-dimensional separable Banach space cannot be Polish under any topology as it has the cardinality $2^{\mathfrak{c}}$. \end{remark}

\subsection*{Closing remark} It would be, of course, desirable to remove the large-cardinal assumption from the statement of Theorem A together with the hypothesis that the filter must be projective. One has to bear in mind that the class of projective filters is rather small, as $\Delta$ has \emph{only} continuum many projective subsets, yet already the set of all maximal filters (ultrafilters) has cardinality $2^{\mathfrak{c}}$, so there appears to be a whole grey area of filters to which our methods are not applicable. 

\subsection*{Acknowledgements} We are indebted to Andr\'es Caicedo for drawing our attention to the Shelah--Woodin theorem and explaining to us various nuances concerning $L(\mathbb R)$.

\bibliographystyle{plain}

\end{document}